\documentclass[12pt]{article}
\usepackage{amsfonts}
\usepackage{amssymb}
\usepackage{amsmath}
\usepackage{fullpage}
\usepackage{xy}
\xyoption{all}

\begin{document}



\newtheorem{thm}{Theorem}[section]
\newtheorem{prop}[thm]{Proposition}
\newtheorem{lem}[thm]{Lemma}
\newtheorem{df}[thm]{Definition}
\newtheorem{cor}[thm]{Corollary}
\newtheorem{rem}[thm]{Remark}
\newtheorem{ex}[thm]{Example}
\newenvironment{proof}{\medskip
\noindent {\bf Proof.}}{\hfill \rule{.5em}{1em}\mbox{}\bigskip}

\def\eps{\varepsilon}

\def\GlR#1{\mbox{\it Gl}(#1,\R)}
\def\glR#1{{\frak gl}(#1,\R)}
\def\GlC#1{\mbox{\it Gl}(#1,\C)}
\def\glC#1{{\frak gl}(#1,\C)}
\def\Gl#1{\mbox{\it Gl}(#1)}
\def\Sp{\mbox{\it Sp}}
\def\ASp{\mbox{\it ASp}}

\def\ov{\overline}
\def\ot{\otimes}
\def\und{\underline}
\def\w{\wedge}
\def\ra{\rightarrow}
\def\lra{\longrightarrow}

\def\rk{\mbox{\rm rk}}
\def\mod{\mbox{ mod }}

\newcommand{\al}{\alpha}
\newcommand{\be}{\beta}
\newcommand{\ga}{\gamma}
\newcommand{\la}{\lambda}
\newcommand{\om}{\omega}
\newcommand{\Om}{\Omega}
\renewcommand{\th}{\theta}
\newcommand{\Th}{\Theta}

\newcommand{\lb}{\langle}
\newcommand{\rb}{\rangle}

\def\pair#1#2{\left<#1,#2\right>}
\def\big#1{\displaystyle{#1}}

\def\N{{\Bbb N}}
\def\Z{{\Bbb Z}}
\def\R{{\Bbb R}}
\def\C{{\Bbb C}}
\def\F{{\Bbb F}}
\def\CP{{\Bbb C} {\Bbb P}}
\def\HP{{\Bbb H} {\Bbb P}}
\def\P{{\Bbb P}}
\def\Q{{\Bbb Q}}
\def\H{{\Bbb H}}
\def\K{{\Bbb K}}

\def\FV{{\frak F}_V}

\def\so{{\frak {so}}}
\def\co{{\frak {co}}}
\def\sl{{\frak {sl}}}
\def\su{{\frak {su}}}
\def\sp{{\frak {sp}}}
\def\asp{{\frak {asp}}}
\def\csp{{\frak {csp}}}
\def\spin{{\frak {spin}}}
\def\stab{{\frak {stab}}}
\def\g{{\frak g}}
\def\h{{\frak h}}
\def\s{{\frak s}}
\def\k{{\frak k}}
\def\l{{\frak l}}
\def\m{{\frak m}}
\def\n{{\frak n}}
\def\t{{\frak t}}
\def\u{{\frak u}}
\def\z{{\frak z}}
\def\r{{\frak {rad}}}
\def\p{{\frak p}}
\def\L{{\frak L}}
\def\X{{\frak X}}
\def\gl{{\frak {gl}}}
\def\hol{{\frak {hol}}}
\renewcommand{\frak}{\mathfrak}
\renewcommand{\Bbb}{\mathbb}
\def\srk{\mbox{\rm srk}}

\def\hook{\mbox{}\begin{picture}(10,10)\put(1,0){\line(1,0){7}}
  \put(8,0){\line(0,1){7}}\end{picture}\mbox{}}
\newcommand{\cyclic}{\mathop{\kern0.9ex{{+}\kern-2.2ex\raise-.29ex%
      \hbox{\Large\hbox{$\circlearrowright$}}}}\limits}

\def\be{\begin{equation}}
\def\ee{\end{equation}}
\def\bi{\begin{enumerate}}
\def\ei{\end{enumerate}}
\def\ba{\begin{array}}
\def\ea{\end{array}}
\def\bea{\begin{eqnarray}}
\def\eea{\end{eqnarray}}
\def\ben{\begin{enumerate}}
\def\een{\end{enumerate}}
\def\linebr{\linebreak}

\title{Extrinsically Immersed Symplectic Symmetric Spaces}

\author{Tom Krantz, Lorenz J. Schwachh\"ofer\thanks{Both authors were supported by the
Schwerpunktprogramm Differentialgeometrie of the Deutsche
Forschungsgemeinschaft. The second author thanks the Max-Planck-Institut f\"ur Mathematik 
in den Naturwissenschaften in Leipzig, Germany, for its hospitality.}}

\date{ }

\maketitle

\begin{abstract}
Let $(V, \Om)$ be a symplectic vector space and let $\phi: M \ra V$ be a symplectic immersion. 
We show that $\phi(M) \subset V$ is locally an extrinsic symplectic symmetric space (e.s.s.s.) 
in the sense of \cite{CGRS} if and only if the second fundamental form of $\phi$ is parallel.

Furthermore, we show that any symmetric space which admits an immersion as an e.s.s.s. also 
admits a {\em full} such immersion, i.e., such that $\phi(M)$ is not contained in a proper affine 
subspace of $V$, and this immersion is unique up to affine equivalence.

Moreover, we show that any extrinsic symplectic immersion of $M$ factors through to the full
one by a symplectic reduction of the ambient space. In particular, this shows that the full 
immersion is characterized by having an ambient space $V$ of minimal dimension.
\end{abstract}

\section{Introduction}

Ever since their introduction by \'E. Cartan (\cite{Car}), symmetric spaces have been studied 
intensely from various viewpoints.
In \cite{Ferus}, D. Ferus introduced the notion of an {\em extrinsic symmetric space} which is 
a Riemannian symmetric space admitting an embedding into a Euclidean vector space 
such that the geodesic reflection at each $p \in M$ is the restriction of the reflection in the 
normal space of $M$ in $p$. In fact, Ferus gave a classification of all Riemannian extrinsic 
symmetric spaces, showing that most Riemannian symmetric spaces admit such an embedding.
Combining this with his results from \cite{Fe1}, he showed that a submanifold of Euclidean 
space is locally extrinsically symmetric if and only if its shape operator is parallel.

Evidently, the concept of extrinsic symmetric spaces may be generalized to other classes as well. 
See \cite{Kath, Kim} for results on extrinsic pseudo-Riemannian symmetric spaces and \cite{GS} 
for extrinsic CR-symmetric spaces. In both cases, many examples of such embeddings are given, 
but a classification appears to be out of reach.

In \cite{CGRS}, {\em extrinsic symplectic symmetric spaces} (e.s.s.s.) were considered for the 
first time, and their basic algebraic and geometric properties were described. However, in contrast 
to the aforementioned cases, only few types of e.s.s.s. are known, all of which have a $3$-step 
nilpotent transvection group.

In this article, we further investigate e.s.s.s. In fact, we slightly generalize this notion to {\em 
extrinsic symplectic immersions }by which we mean an immersion $\phi: M \ra V$ of a symmetric 
space $M$ into the symplectic vector space $(V, \Om)$ such that $\phi^*(\Om)$ is 
non-degenerate and such that for all $p \in M$
\[
\phi \circ s_p = \sigma_{T_p} \circ \phi,
\]
where $s_p: M \ra M$ is the symmetric reflection at $p$, and $\sigma_{T_p}: V \ra V$ is the 
reflection in the affine normal space of $T_p := d\phi(T_pM)$. If $\phi$ is an embedding, then 
$\phi(M) \subset V$ is an e.s.s.s. While we do not know of any example of an extrinsic 
symplectic immersion which is {\em not} an embedding, we nevertheless use this notion as it 
allows us to state our results more elegantly. As a first result, we obtain in analogy to the 
aforementioned result from \cite{Fe1}:

\

\noindent
{\bf Theorem~A} {\em Let $(V, \Om)$ be a symplectic vector space and let $\phi: M \ra V$ be a
symplectic immersion, i.e., such that $\phi^*(\Om)$ is non-degenerate.

Then the shape operator of $\phi$ is parallel if and only if there is an extrinsic symplectic 
immersion $\hat \phi: \hat M \ra V$ of a symmetric space $\hat M$ and a local diffeomorphism 
$\eps: M \ra \hat M$ such that $\phi = \hat \phi \circ \eps$.}

\

In contrast to the Riemannian case, e.s.s.s. which are intrinsically equivalent do not need to be 
extrinsically equivalent. For instance, a symplectic vector space, regarded as a flat symmetric 
space, can be embedded non-affinely into a higher-dimensional vector space as an e.s.s.s. 
(\cite{CGRS}). It is precisely this non-uniqueness which we want to illuminate.

An extrinsic symplectic immersion $\phi: M \ra V$ is called {\em full} if there is no proper affine 
subspace of $V$ containing $\phi(M)$.

\

\noindent
{\bf Theorem~B} {\em Let $M$ be a symmetric space admitting an extrinsic symplectic 
immersion. Then $M$ also admits a full extrinsic symplectic immersion which is unique up to 
affine equivalence.}

\

We point out that Theorem~B reveals a feature of e.s.s.s. which is not valid for pseudo-Riemannian 
extrinsic symmetric spaces. Indeed, there are full extrinsic pseudo-Riemannian symmetric spaces 
which are intrinsically equivalent but with ambient spaces of different dimensions; see e.g. 
examples 7.3 and 7.4 in \cite{Kath}.

Consider an extrinsic symplectic immersion $\phi: M \ra V$ into the symplectic space $(V, \Om)$, 
and let $V' \subset V$ be the smallest affine subspace containing $\phi(M)$. Assuming that $0 \in 
\phi(M)$ we may assume that $V'$ is a linear. Let $(V_0, \Om_0)$ be the symplectic reduction of 
$V'$, i.e., $V_0 := V'/N$ where $N \subset V'$ is the null space of $\Om|_{V'}$, and $\pi^*(\Om_0) 
= \Om|_{V'}$ for the canonical projection $\pi: V' \ra V_0$.

\

\noindent
{\bf Theorem C} {\em Let $\phi: M \ra V$ be an extrinsic symplectic immersion with $0 \in \phi(M)$, 
let $V' \subset V$, $V_0 = V'/N$ and $\pi: V' \ra V_0$ be as in the preceding paragraph. Then 
$\phi_0 := \pi \circ \phi: M \ra V_0$ is the full extrinsic symplectic immersion.}

\

\noindent
As an immediate consequence of Theorem~C we get

\

\noindent
{\bf Corollary D} {\em An extrinsic symplectic immersion $\phi: M \ra V$ is full if and only if the 
ambient space $V$ is of smallest possible dimension}.

\

This paper is structured as follows. After setting up our notation in section~\ref{sec:prelim}, we 
define in section~\ref{sec:immersions} the notion of an extrinsic symplectic immersion and discuss 
their relation to extrinsic symplectic morphisms of Lie algebras. In section~\ref{sec:parallel}, we 
prove Theorem~A and finally, section~\ref{sec:minimal} is devoted to the proofs of the remaining 
statements.

\section{Preliminaries} \label{sec:prelim}

\subsection{Symmetric spaces} \label{sec:symmetric-spaces}

We recall some basic facts on symmetric spaces which can be found e.g. in \cite{Hel}.

Let $G$ be a connected Lie group with Lie algebra $\g$, and let $\sigma: G \ra G$ be an 
involution, i.e., a Lie group homomorphism satisfying $\sigma^2 = Id$. Then $Fix(\sigma) := \{ g \in 
G \mid \sigma(g) = g\}$ is a closed subgroup of $G$ with Lie algebra $\k = \{ x \in \g \mid 
d\sigma(x) = x\}$. Let $K \subset G$ be a closed $\sigma$-invariant subgroup with Lie algebra 
$\k$ (i.e., the identity component of $K$ coincides with that of $Fix(\sigma)$). Then the coset 
space $M := G/K$ is a manifold. Since $\sigma: G \ra G$ is an involution, so is $d\sigma: \g \ra 
\g$, hence we may decompose
\be \label{eq:dsigma}
\g = \k \oplus \p,
\ee
where $d\sigma|_\k = Id$ and $d\sigma|_\p = - Id$, and hence we get for the Lie brackets
\be \label{eq:symmetric-brackets}
[\k, \k] \subset \k, \mbox{\hspace{1cm}} [\k, \p] \subset \p, \mbox{\hspace{1cm}} [\p, \p] 
\subset \k.
\ee
Conversely, given a direct sum decomposition $\g = \k \oplus \p$ satisfying 
(\ref{eq:symmetric-brackets}), then the map $d\sigma|_\k = Id$ and $d\sigma|_\p = - Id$ is a Lie 
algebra homomorphism, hence it is the differential of an involution $\sigma: G \ra G$, if $G$ is the 
simply connected Lie group with Lie algebra $\g$ and hence gives rise to a coset space $M 
= G/K$ as described above.

By definition of $K$, $\sigma$ induces an involution $s_0: M \ra M$ by $s_0(gK) := \sigma(g)K$. 
For $p = gK \in M$, we define the {\em symmetry at $p$} $s_p: M \ra M$ by $s_p := L_g \circ s_0 
\circ L_{g^{-1}}$, where $L_g: M \ra M$ denotes the canonical left action by $g \in G$. Then one 
easily verifies the well-definedness of $s_p$ as well as the following properties:
\bi
\item
$s_p^2 = Id_M$, and $p$ is an isolated fixed point of $s_p$ for all $p \in M$.
\item
$s_p \circ s_q \circ s_p = s_{s_p(q)}$ for all $p, q \in M$.
\ei

It is known that the group $\hat G := \langle s_p \mid p \in M \rangle \subset \mbox{\em Diff}(M)$ 
can be presented by $\{s_p \mid p \in M\}$ subject to these relations. Moreover, the identity 
component of $\hat G$ is known to be a Lie group, called the {\em transvection group of $M$}, 
and its Lie algebra is $[\p, \p] \oplus \p \lhd \g$.

A manifold with smooth maps $s_p$ for all $p \in M$ satisfying these relations is called a {\em 
symmetric space}. Thus, we have described above how a connected Lie group $G$ with an 
involution $\sigma: G \ra G$ induces the structure of a symmetric space on $M := G/K$ for a 
closed $\sigma$-invariant subgroup $K \subset G$ with Lie algebra $\k$.

This process can be reverted. Namely, if $M$ is a symmetric space, let $G$ be the (connected 
component of the) transvection group, and for a fixed $p_0 \in M$ define the involution
\[
\sigma:= AD_{s_{p_0}}: G \ra G,
\]
where $AD$ denotes conjugation in $G$. This is a well defined involution as $\sigma (s_p) = s_{s_{p_0}(p)} \in \hat G$ for all $p$. Since 
$p_0$ is an isolated fixed point of $s_{p_0}$, it follows that $K := Stab(p_0) \subset G$ is indeed 
a $\sigma$-invariant Lie subgroup with Lie algebra $\k$ from (\ref{eq:dsigma}) and hence fits the 
above description. Observe that if we assume $M = G/K$ to be simply connected, then $K$ is 
connected and hence $K = Fix(\sigma)_0$.

\subsection{The curvature of symplectic symmetric spaces} 

A symmetric pair $\g = \k \oplus \p$ is called {\em symplectic}, if there is an $ad_\k$-invariant 
symplectic form $\om$ on $\p$. Likewise, a symmetric space $M$ is called {\em symplectic}, if 
there is a symplectic form on $M$ which is invariant under all symmetries and hence under the 
transvection group.

It is elementary to show that the correspondence between simply connected symplectic 
symmetric spaces and symplectic symmetric pairs described in 
section~\ref{sec:symmetric-spaces} is a bijection. Our task shall be to investigate the curvature of 
a symplectic symmetric space and to give a criterion when such a symplectic symmetric space 
admits an extrinsic symplectic immersion.

If $\g = \k \oplus \p$ is a symplectic symmetric pair, then $(ad_\k)|_\p$ yields a homomorphism 
$ad: \k \ra \sp(\p, \om)$, and its {\em curvature} is defined as the map (cf. \cite[IV Thm. 4.2]{Hel})
\[
R : \Lambda^2 \p \longrightarrow \sp(\p, \om), \mbox{\hspace{1cm}} R(x, y) := - ad_{[x, y]}|_\p.
\]
The Jacobi identity of $\g$ is equivalent to the $ad_\k$-equivariance of $R$ and $R$ satisfying the 
{\em first Bianchi identity}, i.e., $R$ being contained in the {\em space of formal curvatures of 
$\sp(\p, \om)$}
\[
K(\sp(\p, \om)) = \{ R \in \Lambda^2 \p^* \ot \sp(\p, \om) \mid R(x, y) z + R(y, z) x + R(z, x) y = 0 
\mbox{ for all $x, y, z \in \p$}\}.
\]
Thus, there is a one-to-one correspondence between transvective symplectic symmetric pairs 
and elements $R \in K(\sp(\p, \om))$ such that $ad_k R = 0$ for all $k = R(x, y) \in \sp(\p, \om)$.

As a representation space of $\sp(\p, \om)$, $K(\sp(\p, \om))$ is isomorphic to 
$\Lambda^2(\sp(\p, \om))$ with an explicit isomorphism given by (\cite[Lemma 1.8]{CGRS})
\be \label{eq:identify-curvature}
R_{A_1 \w A_2}(x, y) := \frac12 ( (A_1 x) \circ (A_2 y) - (A_1 y) \circ (A_2 x)) \in S^2(\p) \cong 
\sp(\p, \om),
\ee
where the identification $S^2(\p) \cong \sp(\p, \om)$ is given by
\[
(x \circ y) \cdot z := \om(x, z) y + \om(y, z) x.
\]
Observe that our isomorphism differs from the one given in (\cite[Lemma 1.8]{CGRS}) by a factor 
$\frac12$, but this will be convenient later on in Lemma~\ref{lem:curvature}.

Thus, we can find an element
\be \label{eq:define-R}
{\cal R} = \sum_{i,j} a_{ij} A_i \w A_j \in \Lambda^2 \sp(\p, \om),
\ee
where $(a_{ij})$ is a non-degenerate skew-symmetric matrix and $A_i \in \sp(\p, \om)$ are linearly independent elements (which of course do not need to form a basis) such that
\begin{eqnarray*}
-ad_{[x, y]} & = & {\cal R}(x, y) := \sum_{i,j} a_{ij} R_{A_i \w A_j}(x, y)\\
& = & \frac12 \sum_{i,j} a_{ij} ( (A_i x) \circ (A_j y) - (A_i y) \circ (A_j x))\\
& = & \sum_{i,j} a_{ij} (A_i x) \circ (A_j y)
\end{eqnarray*}
for all $x, y \in \p$. We may write (\ref{eq:define-R}) also as
\[
{\cal R} = \sum_i A_i \w A^i, \mbox{ where } A^i := \sum_j a_{ij} A_j.
\]
The $(ad_\k)|_\p$-equivariance of ${\cal R}$ is equivalent to
\[
0 = ad_k {\cal R} = 2 \sum_{i,j} a_{ij} [ad_k, A_i] \w A_j = 2 \sum_i [ad_k, A_i] \w A^i
\]
for all $k \in \k$, which implies

\begin{prop}
Let $\k \subset \sp(\p, \om)$ and let ${\cal R} \in \Lambda^2 \sp(\p, \om) \cong K(\sp(\p, \om))$ be 
given as in (\ref{eq:define-R}). Then $\g := \k \oplus \p$ becomes a Lie algebra and hence a 
symmetric pair whose curvature is represented by ${\cal R}$ if and only if there are $\kappa_{ij} = 
\kappa_{ji} \in \k^*$ such that 
\be \label{eq:invariance-curvature}
[ad_k, A_i] = \sum_j \kappa_{ij}(k) A^j \mbox{ for all $k \in \k$}.
\ee
\end{prop}

\section{Extrinsic symplectic immersions} \label{sec:immersions}

In this section, we shall generalize the notion of an {\em extrinsic symplectic symmetric space} 
from \cite{CGRS} to extrinsic symplectic immersions. We begin by introducing some notation.

We define $ASp(V, \Om) = Sp(V, \Om) \ltimes V$ to be the group of affine symplectic 
transformations of $V$ and let
\[ 
Lin: ASp(V, \Om) \longrightarrow Sp(V, \Om)
\]
be the canonical projection $ASp(V) \ra ASp(V, \Om)/V \cong Sp(V, \Om)$ which is thus a 
homomorphism.

\begin{df} Let $(V, \Om)$ be a symplectic vector space, and let $(W, p)$ be a pair of a subspace 
$W \subset V$ such that $\Om|_W$ is non-degenerate, and a point $p \in V$. The {\em reflection 
in $(W, p)$} is the involution $\sigma_{(W, p)} \in ASp(V, \Om)$ given by
\[
\sigma_{(W, p)}(p) = p, \mbox{ } Lin(\sigma_{(W,p)})|_W = -Id_W, \mbox{Êand } 
Lin(\sigma_{(W,p)})|_{W^\perp} = Id_{W^\perp}.
\]
Here, $\mbox{}^\perp$ refers to the orthogonal complement w.r.t. $\Om$.
\end{df}

\noindent
Note that the definition of the reflections implies that for all $g \in ASp(V, \Om)$ we have
\be \label{eq:commute-reflections}
AD_g (\sigma_{(W, p)}) = \sigma_{(Lin(g) \cdot W, g \cdot p)}.
\ee

\begin{df}
Let $M$ be a symmetric space and $(V, \Om)$ a symplectic vector space. An {\em extrinsic 
symplectic immersion of $M$} is an immersion $\phi: M \ra V$ such that for all $p \in M$ we have
\bi
\item
$\phi$ is a {\em symplectic immersion}, i.e., $\phi^*(\Om)$ is non-degenerate or, equivalently, the 
restriction of $\Om$ to $T_p := d\phi_p(T_p M)$ is non-degenerate,
\item
$\phi \circ s_p = \sigma_{(T_p, \phi(p))} \circ \phi$.
\ei
We define $N_p := (T_p)^\perp$ to be the $\Om$-orthogonal complement, so that we have the 
decompositions $V = T_p \oplus N_p$ for all $p \in M$.
\end{df}
Thus, there is a well defined homomorphism
\be \label{eq:def-i-phi}
\imath_\phi: G \longrightarrow ASp(V, \Om), \mbox{ } \imath_\phi(s_p) := \sigma_{(T_p, \phi(p))},
\ee
where $G$ is the transvection group of $M$, since this definition respects the relations for $s_p$. 
Then evidently, we have the relation
\be \label{eq:induced-homom}
\imath_\phi(g) \cdot \phi(p) = \phi(g \cdot p) \mbox{ for all $g \in G$ and $p \in M$}.
\ee
It follows that
\begin{eqnarray*}
Lin(\imath_\phi(g))(T_p) & = & d(\imath_\phi(g) \circ \phi) (T_pM)\\
& = & d(\phi \circ L_g) (T_pM) \mbox{ by (\ref{eq:induced-homom})}\\
& = & d\phi (T_{g \cdot p}M) = T_{g \cdot p},
\end{eqnarray*}
where $L_g: M \ra M$ denotes the action of $g \in G$ on $M$. That is, (\ref{eq:induced-homom}) 
implies
\be \label{eq:relate-T_p's}
Lin(\imath_\phi(g))(T_p) = T_{g \cdot p} \mbox{ and } Lin(\imath_\phi(g))(N_p) = N_{g \cdot p} 
\mbox{ for all $g \in G$ and $p \in M$}.
\ee

Let us now write $M = G/K$ which induces the Lie algebra decomposition (\ref{eq:dsigma}), and 
suppose that $\phi: M \ra (V, \Om)$ is an extrinsic symplectic immersion with $\phi(p_0) = 0$ for 
$p_0 := eK \in M$. Then by (\ref{eq:induced-homom}) the diagram
\be \label{diagram}
\xymatrix{
G \ar[d]^{pr} \ar[r]^{\imath_\phi} & ASp(V, \Om) \ar[d]^{\pi_0}\\
M \ar[r]^{\phi} & V
}
\ee
commutes, where $pr: G \ra M = G/K$ and $\pi_0: ASp(V, \Om) \ra ASp(V, \Om)/Sp(V, \Om) \cong 
V$ are the canonical projections, i.e., $\pi_0(g) = g \cdot 0$ for all $g \in ASp(V, \Om)$. Moreover, 
we consider the $\Om$-orthogonal decomposition
\be \label{eq:decompose-V}
V = W_1 \oplus W_2,
\ee
where $W_1 := T_{p_0} = d\phi(T_{p_0}M)$ is a non-degenerate subspace by hypothesis. Let 
$\sigma_0 := \sigma_{(W_1,0)}$, so that $\sigma_0 = \imath_\phi(s_{p_0})$ and hence by 
(\ref{eq:def-i-phi})
\be \label{eq:commute-sigma}
\imath_\phi \circ \sigma = \imath_\phi \circ AD_{s_{p_0}} = AD_{\sigma_0} \circ \imath_\phi,
\ee
whose derivative reads
\be \label{eq:der-commute-sigma}
d\imath_\phi \circ d\sigma = Ad_{\sigma_0} \circ d\imath_\phi,
\ee
where $d\sigma: \g \ra \g$ is the involution with eigenspaces $\k$ and $\p$, and where $Ad$ 
denotes the differential of the conjugation map $AD$. If we decompose 
$d\imath_\phi = \Lambda + \tau: \g \ra \asp(V, \Om) = \sp(V, \Om) \oplus V$, then 
(\ref{eq:der-commute-sigma}) is equivalent to
\be \label{eq:der-commute-sigma2}
\Lambda \circ d\sigma = Ad_{\sigma_0} \circ \Lambda \mbox{\hspace{5mm} and \hspace{5mm}} 
\tau \circ d\sigma = \sigma_0 \circ \tau.
\ee
Since $AD_{\sigma_0}: Sp(V, \Om) \ra Sp(V, \Om)$ is an involution with differential 
$Ad_{\sigma_0}: \sp(V, \Om) \ra \sp(V, \Om)$, we get the corresponding symmetric space 
decomposition
\be \label{eq:define tilde k/p}
\sp(V, \Om) = \tilde \k \oplus \tilde \p, \mbox{ where } \ba{l} \tilde \k = \{ x \in \sp(V, \Om) \mid x W_i 
\subset W_i\} \mbox{ and }\\ \tilde \p = \{ x \in \sp(V, \Om) \mid x W_i \subset W_{i+1}\}, 
\mbox{ taking $i \mod 2$.}\ea
\ee

\begin{df} \label{def:realize-algebra}
Let $\g = \k \oplus \p$ be a symmetric pair and $(V, \Om)$ a symplectic vector space. We call a 
Lie algebra homomorphism
\[
d\imath: \g \longrightarrow \asp(V, \Om), d\imath(x) := \Lambda(x) + \tau(x) \in \sp(V, \Om) \oplus V
\]
an {\em extrinsic symplectic morphism of $\g$} if there is a symplectic orthogonal decomposition 
$V = W_1 \oplus W_2$ such that
\bi
\item
$\k = \ker(\tau)$ and $\tau: \p \ra W_1$ is a linear isomorphism,
\item
We have $\Lambda \circ d\sigma = Ad_{\sigma_0} \circ \Lambda$ for the involution $d\sigma: \g 
\ra \g$ and $Ad_{\sigma_0}: \sp(V, \Om) \ra \sp(V, \Om)$, or, equivalently, with the splitting from 
(\ref{eq:define tilde k/p})
\[
\Lambda(\k) \subset \tilde \k \mbox{ and } \Lambda(\p) \subset \tilde \p.
\]
\ei
\end{df}

\noindent
{\bf Remark.} The reader who is familiar with \cite{CGRS} will note that $d\imath = \Lambda + 
\tau$ is an extrinsic symplectic morphism of $\g$ if and only if $\Lambda$ satisfies the conditions 
in \cite[Lemma 1.5]{CGRS} when identifying $\p \cong W_1$ by the isomorphism $\tau$.

\

We now wish to describe how the notions of the extrinsic symplectic immersion of a symmetric 
space $M$ and extrinsic symplectic morphisms of a symmetric pair are related. 

\begin{thm} \label{thm:extrinsic-group-algebra} (cf. \cite[Lemmas 1.5 and 1.7]{CGRS})
\bi
\item
Let $M = G/K$ be a symmetric space and let $\g = \k \oplus \p$ be the corresponding symmetric 
pair. Let $\phi: M \ra (V, \Om)$ be an extrinsic symplectic immersion of $M$ with $\phi(eK) = 0$, 
and let $\imath_\phi: G \ra ASp(V, \Om)$ be the Lie group homomorphism from (\ref{eq:def-i-phi}).

Then $d\imath_\phi: \g \ra \asp(V, \Om)$ is an extrinsic symplectic morphism of the symmetric 
pair $\g = \k \oplus \p$.

\item
Conversely, let $\g = \k \oplus \p$ be a symmetric pair and let $d\imath: \g \ra \asp(V, \Om)$ be an 
extrinsic symplectic morphism.

Then there is a symmetric space $M = G/K$ corresponding to this symmetric pair and an injective 
extrinsic symmetric immersion $\phi: M \ra (V, \Om)$ such that $d\imath = d\imath_\phi$.
\ei
\end{thm}

\begin{proof}
Let $\phi: M \ra V$ be an extrinsic symplectic immersion of $M$ with $\phi(p_0) = 0$ where $p_0 
= eK \in G/K = M$, let $W_1 := T_{p_0} = d\phi(T_{p_0}M)$ and consider the splitting 
(\ref{eq:decompose-V}).

Since (\ref{diagram}) commutes, it follows that $\tau = d(\pi_0 \circ \imath_\phi)_e = 
d(\phi \circ pr)_e$, and since $\phi$ is an immersion, we have $\ker(\tau) = \ker(dpr_e) = \k$. 
Moreover, since $pr: G \ra M$ is a submersion, it follows that $Im(\tau) = 
d(\pi_0 \circ \imath_\phi)(\g) =  d\phi(T_{p_0}M) = W_1$, showing the first property in 
definition~\ref{def:realize-algebra}. The second follows immediately from 
(\ref{eq:der-commute-sigma2}).

To show the second statement, let $\g = \k \oplus \p$ and $G$ be the simply connected Lie group 
with Lie algebra $\g$, let $\sigma: G \ra G$ be the involution with differential $d\sigma$ and 
$\imath: G \ra ASp(V, \Om)$ the homomorphism with differential $d\imath: \g \ra \asp(V, \Om)$. By 
Definition~\ref{def:realize-algebra}, it follows that (\ref{eq:der-commute-sigma2}) holds as 
$\tau(\k) = 0$ and $\tau(\p) = W_1$. Since this is equivalent to (\ref{eq:der-commute-sigma}), 
integration yields that $\imath$ satisfies (\ref{eq:commute-sigma}).

We let $K := \imath^{-1}(Sp(V, \Om)) \subset G$ which is closed and has $\k$ as its Lie algebra 
since $x \in \k$ iff $d(\pi_0 \circ \imath)(x) = 0$. Moreover,  for $k \in K$ we have by 
(\ref{eq:commute-sigma}) $\imath(\sigma(k)) = AD_{\sigma_0} \imath(k) \in Sp(V, \Om)$, since 
$AD_{\sigma_0}$ leaves $Sp(V, \Om)$ invariant. Therefore, $K \subset G$ is a closed 
$\sigma$-invariant subgroup with Lie algebra $\k$, hence $M := G/K$ is a symmetric space 
corresponding to the symmetric pair $\g = \k \oplus \p$.

There is an induced map $\phi: M \ra ASp(V, \Om)/Sp(V, \Om) = V$ for which (\ref{diagram}) 
commutes with $\imath_\phi := \imath$, and $\phi$ is an injective immersion. The commutativity 
of (\ref{diagram}) implies that $\phi(gK) = (\phi \circ pr)(g) = \pi_0(\imath_\phi(g)) = \imath_\phi(g) 
\cdot 0 = \imath_\phi(g) \cdot \phi(eK)$ for all $g \in G$, so that (\ref{eq:induced-homom}) holds, 
which was shown to imply (\ref{eq:relate-T_p's}). Therefore, $T_{gK} = Lin_{\imath(g)}(T_{eK}) = 
Lin_{\imath(g)}(W_1)$, and since $\Om|_{W_1}$ is non-degenerate, $\Om|_{T_p}$ is 
non-degenerate for all $p \in M$, i.e., $\phi: M \ra V$ is a symplectic immersion.

For $g, h \in G$ and $p := gK$ and $q := hK \in M = G/K$ we have
\begin{eqnarray*}
\phi (s_p (q)) & = & \phi (g \sigma(g^{-1}h)K) = (\imath_\phi(g \sigma(g^{-1} h)))(0) 
\mbox{ by (\ref{diagram})}\\
& = & \imath_\phi(g) \cdot \imath_\phi(\sigma(g^{-1} h))(0) = \imath_\phi(g) \cdot 
AD_{\sigma_0}(\imath_\phi(g^{-1}h))(0) \mbox{ by (\ref{eq:commute-sigma})}\\
& = & \imath_\phi(g) \cdot \sigma_0 \cdot \imath_\phi(g)^{-1} \cdot \imath_\phi(h)(0) 
\mbox{ as $\sigma_0(0) = 0$}\\
& = & AD_{\imath_\phi(g)} (\sigma_{(T_{eK}, 0)}) \cdot \phi(q) \mbox{ by (\ref{diagram})}\\
& = & \sigma_{(Lin(\imath_\phi(g)) \cdot T_{eK}, \imath_\phi(g)(0))} \cdot \phi(q)  
\mbox{ by (\ref{eq:commute-reflections})}\\
& = & \sigma_{(T_p, \phi(p))} \cdot \phi(q) \mbox{ by (\ref{eq:relate-T_p's})},
\end{eqnarray*}
and hence,
\[
\phi \circ s_p = \sigma_{(T_p, \phi(p))} \circ \phi
\]
for all $p \in M$ which is the second property for $\phi$ being a symplectic immersion of $M$. 
\end{proof}

Note that if $\iota: (V_1, \Om_1) \ra (V_2, \Om_2)$ is a symplectic isomorphism and $\phi: M \ra 
V_1$ is an extrinsic symplectic immersion with associated homomorphism $\imath_\phi: G \ra 
ASp(V, \Om)$, then $\imath_{\iota \circ \phi} = AD_\iota \circ \imath_\phi$. Thus, from 
Theorem~\ref{thm:extrinsic-group-algebra} we immediately obtain

\begin{cor}
Let $M = G/K$ be a simply connected symmetric space with corresponding Lie algebra 
decomposition $\g = \k \oplus \p$. Then there is a one-to-one correspondence between affine 
equivalence classes of extrinsic symplectic immersions $\phi: M \ra (V, \Om)$ and conjugacy 
classes of extrinsic symplectic morphisms $d\imath: \g \ra \asp(V, \Om)$.
\end{cor}

If $\g = \k \oplus \p$ is a symmetric pair admitting an extrinsic symplectic morphism $d\imath = 
\Lambda + \tau: \g \ra \asp(V, \Om)$, then evidently, $\tau^*(\Om)$ is an $ad_\k$-invariant 
symplectic form on $\p$. Thus, for the splitting of the Lie algebra $\g$ from (\ref{eq:dsigma}) we 
can describe the curvature $R: \Lambda^2 \p \ra \k$ as in (\ref{eq:define-R}). In this language, 
\cite[Lemmas 1.5, 1.7]{CGRS} as well as the equivalence of equations (1.14) and (1.24) in \cite{CGRS} 
yield the following.

\begin{prop}
Let $\g = \k \oplus \p$ be a symplectic symmetric pair with the $ad_\k$-invariant symplectic form 
$\om \in \Lambda^2 \p^*$, let ${\cal R} = \sum_{i,j} a_{ij} A_i \w A_j \in \Lambda^2 \sp(\p, \om)$ 
be its curvature and let $\kappa_{ij} = \kappa_{ji} \in \k^*$ be as in (\ref{eq:invariance-curvature}).
Then $\g$ admits an extrinsic symplectic morphism if and only if
\[
\kappa_{ij}(R(x, y)) = \om((A_i A_j + A_j A_i) x, y) \mbox{ for all $x, y \in \p$ and all $i, j$}
\]
\end{prop}

\section{Parallelity of the second fundamental form} \label{sec:parallel}

Let $\phi: M \ra (V, \Om)$ be a symplectic immersion of a manifold $M$ into a symplectic vector 
space, i.e., $\phi^*(\Om)$ is non-degenerate. Since $T_p := d\phi(T_pM) \subset V$ is a 
non-degenerate subspace, we obtain the $\Om$-orthogonal splitting $V = T_p \oplus N_p$ for 
all $p \in M$. Thus, we have the pullback vector bundles
\[
{\cal T} := \{ (p, v) \in M \times V \mid v \in T_p \} \mbox{ and } {\cal N} := \{ (p, v) \in M \times V \mid 
v \in N_p \}.
\]
Evidently, ${\cal T} = TM$ is the tangent bundle of $M$ as $\phi$ is an immersion, and ${\cal N} 
\ra M$ is called the {\em normal bundle of $\phi$}. For the direct sum of these vector bundles, we 
have
\be \label{eq:splitting-VB}
M \times V = {\cal T}Ê\oplus {\cal N}.
\ee
The canonical flat connection $\mathring \nabla$ on $V$ induces connections on $TM$ and 
${\cal N}$ by the definitions
\[
\nabla^{\cal T}_XY := (\mathring\nabla_XY)_{\cal T} \mbox{\hspace{3mm}Êand \hspace{3mm}}
\nabla^{\cal N}_X\xi := (\mathring\nabla_X\xi)_{\cal N},
\]
where the subscripts refer to the components of the vector bundle splitting (\ref{eq:splitting-VB}). 
Moreover, the {\em second fundamental formÊ} and the {\em shape operator of $\phi$ }are 
defined as
\begin{equation}
\alpha(X,Y) := (\mathring\nabla_XY)_{\cal N} \mbox{\hspace{3mm}Êand \hspace{3mm}} 
A_\xi X :=  (\mathring\nabla_X\xi)_{\cal T},
\end{equation}
which are related by the identity
\be \label{eq:shape-2nd}
\Om(\al(X, Y), \xi) = \Om(A_\xi X, Y).
\ee
Since $\al$ is symmetric, it follows that $A_\xi \in \sp(T_p, \Om)$.

\

\noindent {\bf Proof of Theorem A}.
If $\phi: M \ra V$ is an extrinsic symplectic immersion, then $s_p: M \ra M$ can be extended to 
bundle maps $s_p: TM \ra TM$ and $s_p: {\cal N} \ra {\cal N}$ in a canonical way, and these 
maps preserve the connections $\nabla^{\cal T}$ and $\nabla^{\cal N}$. Therefore, for all $X, Y, 
Z \in T_p$ we have
\[
s_p((\nabla_X \al)(Y, Z)) = (\nabla_{s_p(X)} \al)(s_p Y, s_p Z) =  (\nabla_{-X} \al)(- Y, - Z) = - 
(\nabla_X \al)(Y, Z).
\]
On the other hand, $s_p((\nabla_X \al)(Y, Z)) = (\nabla_X \al)(Y, Z)$ as $(\nabla_X \al)(Y, Z) \in 
N_p$, so that
\linebreak
$(\nabla_X \al)(Y, Z) = 0$ and hence, $\al$ is parallel. This evidently also holds if $M \ra \hat M 
\ra V$ is the composition of a local diffeomorphism and an extrinsic symplectic immersion.

Conversely, let $\phi: M \ra (V, \Om)$ be a symplectic immersion such that $\al$ is parallel. Fix 
$p_0 \in M$ and assume w.l.o.g. that $\phi(p_0) = 0$. We consider the $\Om$-orthogonal 
decomposition $V = W_1 \oplus W_2$ with $W_1 := d\phi_{p_0}(T_{p_0}M)$. As in  \cite{Fe1}, 
we define the map
\[
\Lambda: W_1 \longrightarrow \sp(V, \Om), \mbox{\hspace{5mm}} \Lambda(x) (y + \xi) = A_\xi x + 
\alpha(x,y),
\]
where $x, y \in W_1$ and $\xi \in W_2$. By (\ref{eq:shape-2nd}) and the symmetry of $\al$, we 
have indeed $\Lambda(x) \in \sp(V, \Om)$, and in fact, $\Lambda(x) \in \tilde \p$ with the splitting 
$\sp(V, \Om) = \tilde \k \oplus \tilde \p$ from (\ref{eq:define tilde k/p}). Now we can apply verbatim 
the proof in \cite{Fe1} of the corresponding statement for Riemannian immersions to the present 
case to see that
\[
\g := \{R(x, y) \mid x, y \in W_1\} \oplus \{ \Lambda(x) + x \mid x \in W_1\} =: \k \oplus \p \subset \asp(V, \Om)
\]
is a Lie subalgebra and hence, the inclusion map $d\imath: \g \hookrightarrow \asp(V, \Om)$ is an extrinsic symplectic morphism. By Theorem~\ref{thm:extrinsic-group-algebra}, there is an injective extrinsic symplectic 
immersion
\[
\hat \phi: \hat M \longrightarrow V
\]
of a symmetric space $\hat M$ with $d\imath_{\hat \phi} = d\imath$, and $\hat \phi(\hat p_0) = 0 = 
\phi(p_0)$ and $d\hat \phi(T_{p_0}\hat M) = W_1 = d\phi(T_{p_0}M)$ for some $\hat p_0 \in \hat 
M$. Therefore, the second fundamental form of $\phi$ and $\hat \phi$ at $0 \in V$ coincide, and 
since both have parallel second fundamental form, it follows that their second fundamental forms 
coincide and hence, $\phi(M) \subset \hat \phi(\hat M)$. As $\hat \phi$ is injective, it follows that 
$\phi = \hat \phi \circ \eps$ for some map $\eps: M \ra \hat M$, and clearly, $\eps$ is a local 
diffeomorphism.
{\hfill \rule{.5em}{1em}\mbox{}\bigskip}

\section{Full extrinsic symplectic immersions} \label{sec:minimal}

Suppose that $d\imath = \Lambda + \tau: \g \ra \asp(V, \Om)$ is an extrinsic symplectic morphism, 
and let $V = W_1 \oplus W_2$ be the $\Om$-orthogonal decomposition with $W_1 = \tau(\p)$, 
and we let $\om := \tau^*(\Om)$ be the symplectic form on $\p$ induced by this morphism. We 
define the map
\be \label{eq:define-C}
{\cal A} : W_2 \longrightarrow \sp(\p, \om), \mbox{\hspace{1cm}} \tau({\cal A}(\xi) (x)) := 
\Lambda(x) \xi \in W_1.
\ee
Indeed, ${\cal A}(\xi) \in \sp(\p, \om)$ as
\begin{eqnarray*}
\om({\cal A}(\xi)x, y) - \om({\cal A}(\xi)y, x) & = & (\tau^*\Om)(\tau^{-1}Ê\Lambda(x) \xi, y) - 
(\tau^*\Om)(\tau^{-1} \Lambda(y) \xi, x)\\
& = & \Om(\Lambda(x) \xi, \tau(y)) - \Om(\Lambda(y) \xi, \tau(x))\\
& = & \Om(\Lambda(x) \tau(y) - \Lambda(y) \tau(x), \xi) = 0,
\end{eqnarray*}
since $\sp(V, \Om) \ni [\Lambda(x) + \tau(x), \Lambda(y) + \tau(y)] = [\Lambda(x), \Lambda(y)] + 
(\Lambda(x) \tau(y) - \Lambda(y) \tau(x))$, so that $\Lambda(x) \tau(y) - \Lambda(y) \tau(x) = 0$ 
for all $x, y \in \p$.

\begin{lem} \label{lem:C-equivariant}
The map ${\cal A}: W_2 \ra \sp(\p, \om)$ from (\ref{eq:define-C}) is $\k$-equivariant; i.e., for all $k 
\in \k$ and $\xi \in W_2$, the following identity holds:
\[
{\cal A}(\Lambda(k) \xi) = [k, {\cal A}(\xi)].
\]
\end{lem}

\begin{proof}
This follows from the following calculation for $x \in \p$, $k \in \k$ and $\xi \in W_2$:
\begin{eqnarray*}
\tau({\cal A}(\Lambda(k) \xi)(x)) & = & \Lambda(x) \Lambda(k) \xi = \Lambda(k) \Lambda(x) \xi - 
[\Lambda(k), \Lambda(x)] \xi\\
& = & \Lambda(k) \tau({\cal A}(\xi) x) - \Lambda(k\ x) \xi = \tau(k\ {\cal A}(\xi) x) - \tau({\cal A}(\xi) 
k\ x)\\
& = & \tau([k, {\cal A}(\xi)] (x)).
\end{eqnarray*}
\vspace{-15mm}

\end{proof}

\begin{lem} \label{lem:curvature} {\em (cf. \cite[(1.26)]{CGRS})}
Let $d\imath: \g \ra \asp(V, \Om)$ be an extrinsic symplectic morphism, let $V = W_1 \oplus W_2$ 
and ${\cal A} : W_2 \ra \sp(\p, \om)$ be as above. Let $\Lambda^2{\cal A}: \Lambda^2 W_2 \ra 
\Lambda^2 \sp(\p, \om)$ be the canonical extension. Then the curvature ${\cal R}$ of $\g = \k 
\oplus \p$ is represented by 
\[
{\cal R} = (\Lambda^2{\cal A})(\Om|_{W_2})^{-1} \in \Lambda^2(\sp(\p, \om))
\]
with the identification from (\ref{eq:identify-curvature}) and where $(\Om|_{W_2})^{-1} \in 
\Lambda^2 W_2$ is the inverse of $\Om|_{W_2} \in \Lambda^2 W_2^*$. That is, if $\{ e_i \}$ is a 
basis of $W_2$ such that $\Om_{ij}:= \Om(e_i, e_j)$ has the inverse matrix $\Om^{ij}$, and if 
$A_i := {\cal A}(e_i)$, then
\[
{\cal R} = \sum_{i,j} \Om^{ij} R_{A_i \w A_j}.
\]
\end{lem}

Once again, observe that in our convention the identification $\Lambda^2(\sp(V, \Om)) \cong 
K(\sp(V, \om))$ differs from that in (\cite{CGRS}) by a factor $\frac12$.


\begin{df} Let $\g = \k \oplus \p$ be a symplectic symmetric pair, and let ${\cal R} \in 
\Lambda^2(\sp(\p, \om))$ be the element representing the curvature of this space. Then
\[
\srk(\g, \k) := \rk ({\cal R})
\]
is called the {\em symplectic curvature rank}. Here, we consider the rank of ${\cal R}$ regarded 
as a $2$-form.
\end{df}

Note that because of Lemma~\ref{lem:curvature} we have $\dim W_2 \geq \srk(\g, \k)$, and 
hence $\dim V \geq \dim \p + \srk(\g, \k)$ for any extrinsic symplectic morphism.

Let $V' \subset V$ be a linear subspace. We define the Lie group
\[
Stab(V') := \{ g \in ASp(V, \Om) \mid g \cdot V' \subset V'\},
\]
whose Lie algebra is given as
\[
\stab(V') := \{ A + v \in \asp(V, \Om) = \sp(V, \Om) \oplus V \mid v \in V' \mbox{ and } A \cdot V' 
\subset V'\}.
\]
We call an extrinsic symplectic immersion $\phi: M \ra V$  {\em full} if there is no proper affine 
subspace $V' \subsetneq V$ such that $\phi(M) \subset V'$. Likewise, we call an extrinsic 
symplectic morphism $d\imath: \g \ra \asp(V, \Om)$ {\em full}, if there is no proper linear 
subspace $V' \subsetneq V$ such that $d\imath(\g) \subset \stab(V') \subset \asp(V, \Om)$. 
Evidently, $\phi$ is full if and only if $d\imath_\phi$ is.

\begin{prop} \label{prop:minimal}
Let $\g = \k \oplus \p$ be a symmetric pair, $V = W_1 \oplus W_2$ a symplectic vector space, and 
let $d\imath = \Lambda + \tau: \g \ra \asp(V, \Om)$ be an extrinsic symplectic morphism. Consider 
the map ${\cal A}: W_2 \ra \sp(\p, \om)$ from (\ref{eq:define-C}). Then the following are 
equivalent.
\bi
\item
$d\imath$ is full.
\item
$\dim(W_2) = \srk(\g, \k)$, i.e., $\dim V = \dim \p + \srk(\g, \k)$.
\item
${\cal A}$ is injective.
\item
$W_2 = span\{\Lambda(x) \tau(y) \mid x, y \in \p\}$.
\ei
\end{prop}

\begin{proof}
Since $(\Lambda^2{\cal A})(\Om|_{W_2})^{-1} = {\cal R}$ by Lemma~\ref{lem:curvature}, it follows 
that $\rk({\cal A}) \geq \rk({\cal R}) = \srk(\g, \k)$. Thus, if $\dim W_2 = \srk(\g, \k)$, then ${\cal A}$ 
must be injective.

Conversely, if ${\cal A}$ is injective, then $\rk((\Lambda^2{\cal A})(\Om|_{W_2})^{-1}) = \dim W_2$ 
since $\Om|_{W_2}$ is non-de\-ge\-ner\-ate. Thus, $\dim W_2 = \rk({\cal R}) = \srk(\g, \k)$ by 
Lemma~\ref{lem:curvature}, which shows that the second and third statement are equivalent.

Let $W_2' := span\{\Lambda(x) \tau(y) \mid x, y \in \p\} \subset W_2$. Then $\xi \in (W_2')^\perp$ 
if and only if $0 = \Om(\xi, \Lambda(x) \tau(y)) = -\Om(\Lambda(x) \xi, \tau(y)) = 
-\Om(\tau({\cal A}(\xi) x), \tau(y)) = -\om({\cal A}(\xi) x, y)$ for all $x, y \in \p$, using the definition of 
${\cal A}$ in (\ref{eq:define-C}), which implies that $(W_2')^\perp = \ker({\cal A})$, and this shows 
the equivalence of the third and fourth statement.

Let $V' \subset V$ be such that $d\imath(\g) \subset \stab(V') \subset \asp(V, \Om)$. Since 
$\tau(\p) = W_1$, it follows that $W_1 \subset V'$. Thus, $[d\imath(x), \tau(y)] = \Lambda(x) \tau(y) 
\in V'$ for all $x, y \in \p$, i.e., $W_1 \oplus W_2' \subset V'$. If $W_2' = W_2$, then we must have 
$V' = V$ which shows that the fourth statement implies the first.

On the other hand, $W_1 \oplus W_2'$ is invariant under $\Lambda(x)$ for all $x \in \p$ by 
definition, and $W_2'$ is invariant under $\Lambda(k)$ for $k \in \k$, since
\begin{eqnarray*}
\Lambda(k) (\Lambda(x) \tau(y)) & = & [\Lambda(k), \Lambda(x)] \tau(y) + \Lambda(x) \Lambda(k) 
\tau(y) = \Lambda(k\ x) \tau(y) + \Lambda(x) \tau(k\ y) \in W_2'.
\end{eqnarray*}
Thus, $d\imath(\g) \subset \stab(W_1 \oplus W_2')$, so that the first statement implies the fourth.
\end{proof}

\begin{prop} \label{prop:full-equivalent}
Let $\g = \k \oplus \p$ be a symmetric pair with two full extrinsic symplectic morphisms $d\imath_i 
= \Lambda_i + \tau_i: \g \ra \asp(V_i, \Om_i)$ such that $\tau_1^*(\Om_1) = \tau_2^*(\Om_2)$. 
Then $d\imath_1$ and $d\imath_2$ are affinely equivalent, i.e., there is a linear symplectic 
isomorphism $\iota: (V_1, \Om_1) \ra (V_2, \Om_2)$ such that $d\imath_2 = Ad_\iota \circ 
d\imath_1$.
\end{prop} 

\begin{proof}
Consider the $\Om_i$-orthogonal decompositions $V_i = W_1^i \oplus W_2^i$ where $\tau_i: \p 
\ra W_1^i$ is an isomorphism for $i = 1,2$. Then ${\cal A}_i: W_2^i \ra \sp(\p, \om)$ is injective for 
$i = 1,2$ by Proposition~\ref{prop:minimal} and moreover, the images ${\cal A}_i(W_2^i) \subset 
\sp(\p, \om)$ coincide by Lemma~\ref{lem:curvature}. Thus, we may define the linear 
isomorphisms 
\[
\iota_1 := \tau_2 \circ \tau_1^{-1}: W_1^1 \longrightarrow W_1^2, \mbox{ and }Ê\iota_2: 
{\cal A}_2^{-1} \circ {\cal A}_1: W_2^1 \longrightarrow W_2^2.
\]
By our hypothesis, $\tau_i^*(\Om_i) = \om$. Moreover, by Lemma~\ref{lem:curvature} 
$({\cal A}_i^{-1})^*(\Om_i^{-1}) = {\cal R} \in \Lambda^2(\sp(V, \om)$. Thus, 
$(\iota_i)^*(\Om_2|_{W_i^2}) = \Om_1|_{W_i^1}$ for $i = 1,2$, i.e.,
\[
\iota := \iota_1 \oplus \iota_2: (V_1, \Om_1) \longrightarrow (V_2, \Om_2)
\]
is a symplectic isomorphism. Since $\tau_2 = \iota_1 \circ \tau_1$ and ${\cal A}_1 = {\cal A}_2 
\circ \iota_2$, we have for $x, y \in \p$ and $\xi \in W_2^2$
\begin{eqnarray*}
\Om_2(\iota \Lambda_1(x) \iota^{-1} \xi, \tau_2(y)) & = & (\iota^* \Om_2)(\Lambda_1(x) 
\iota_2^{-1} \xi, \iota_1^{-1} \tau_2(y))\\
& = & \Om_1(\tau_1 ({\cal A}_1(\iota_2^{-1}(\xi)) x), \tau_1 (y))\\
& = & (\tau_1^* \Om_1)({\cal A}_2(\xi) x, y) = (\tau_2^* \Om_2)({\cal A}_2(\xi) x, y)\\
& = & \Om_2(\tau_2({\cal A}_2(\xi) x), \tau_2 (y))\\
& = & \Om_2(\Lambda_2(x) \xi, \tau_2 (y)),
\end{eqnarray*}
and since $\iota \Lambda_1(x) \iota^{-1}(W_i^2) \subset W_{i+1}^2$ and $\Lambda_2(x)(W_i^2) 
\subset W_{i+1}^2$ taking indices $\mod 2$, it follows that $\iota \Lambda_1(x) \iota^{-1} = 
\Lambda_2(x)$ for all $x \in \p$. Similarly, we have for $k \in \k$ and $x, y \in \p$
\begin{eqnarray*}
\Om_2(\iota \Lambda_1(k) \iota^{-1} \tau_2(x), \tau_2(y)) & = & 
(\iota^* \Om_2)(\Lambda_1(k) \iota_1^{-1} \tau_2(x), \iota_1^{-1} \tau_2(y))\\
& = & \Om_1(\Lambda_1(k) \tau_1(x), \tau_1(y))\\
& = & \Om_1(\tau_1(k\ x), \tau_1(y))\\
& = & (\tau_1^* \Om_1)(k\ x, y) = (\tau_2^* \Om_2)(k\ x, y)\\
& = & \Om_2(\tau_2(k\ x), \tau_2(y))\\
& = & \Om_2(\Lambda_2(k) \tau_2(x), \tau_2(y)),
\end{eqnarray*}
and since $\iota \Lambda_1(k) \iota^{-1} (W_1^2) \subset W_1^2$ and $\Lambda_2(k) (W_1^2) 
\subset W_1^2$, it follows that $\iota \Lambda_1(k) \iota^{-1}|_{W_1^2} = \Lambda_2|_{W_1^2}$.
Finally, for $k \in \k$ and $\xi \in W_2^2$ we have by Lemma~\ref{lem:C-equivariant}
\begin{eqnarray*}
\iota \Lambda_1(k) \iota^{-1} \xi & = & \iota_2 \Lambda_1(k) \iota_2^{-1} \xi\\
& = & {\cal A}_2^{-1} {\cal A}_1(\Lambda_1(k) (\iota_2^{-1}(\xi))\\
& = & {\cal A}_2^{-1} [k, {\cal A}_1(\iota_2^{-1}(\xi))]\\
& = & {\cal A}_2^{-1} [k, {\cal A}_2(\xi)]\\
& = & {\cal A}_2^{-1} {\cal A}_2 (\Lambda_2(k)(\xi)) = \Lambda_2(k)(\xi),
\end{eqnarray*}
so that $\iota \Lambda_1(k) \iota^{-1}|_{W_2^2} = \Lambda_2|_{W_2^2}$, which completes 
showing that
\[
\Lambda_2 = AD_\iota \circ \Lambda_1 \mbox{ and } \tau_2 = \iota \circ \tau_1,
\]
and therefore, $d\imath_2 = AD_\iota \circ d\imath_1$.
\end{proof}

\noindent {\bf Proof of Theorem~C.} Let $\phi: M \ra V$ be an extrinsic symplectic immersion with 
$0 \in \phi(M)$, let $V' \subset V$ be the smallest subspace containing $\phi(M)$, and let $N 
\subset V'$ be the nullspace of $\Om|_{V'}$. Then $V_0 := V'/N$ has an induced symplectic 
structure $\Om_0$ such that $\pi^*(\Om_0) = \Om|_{V'}$, where $\pi: V' \ra V_0$ is the canonical 
projection.

Since $T_p \subset V'$ for all $p \in M$ and $\Om|_{T_p}$ is non-degenerate, it follows that $N 
\subset T_p^\perp$ and hence, $N \cap T_p = 0$ for all $p \in M$. Thus, $\phi_0 := \pi \circ \phi: M 
\ra V_0$ is an immersion, and $\Om_0|_{T_p^0}$ is non-degenerate, where $T_p^0 = 
d\phi_0(T_pM) = \pi(T_p)$. Moreover, for the reflections we have 
\[
\pi \circ \sigma_{(T_p, \phi(p))} = \sigma_{(T_p^0, \phi_0(p))} \circ \pi
\]
for all $p$, whence $\sigma_{(T_p^0, \phi_0(p))} \circ \phi_0 = \sigma_{(T_p^0, \phi_0(p))} \circ \pi 
\circ \phi = \pi \circ \sigma_{(T_p, \phi(p))} \circ \phi = \pi \circ \phi \circ s_p = \phi_0 \circ s_p$, 
using that $\phi$ is an extrinsic symplectic immersion. Thus, $\phi_0: M \ra V_0$ is an extrinsic 
symplectic immersion as well.

If $\phi_0(M) \subset V_0' \subset V_0$, then $\phi(M) \subset \pi^{-1}(V_0') \subset V'$. Since we 
chose $V' \subset V$ to be the smallest subspace containing $\phi(M)$, it follows that $V_0' = 
V_0$, and this implies that $\phi_0: M \ra V_0$ is full.
{\hfill \rule{.5em}{1em}\mbox{}\bigskip}

Theorem~B from the introduction now follows immediately from combining Theorem~C with 
Proposition~\ref{prop:full-equivalent}, and Corollary~D is an immediate consequence of 
Theorem~C.

\noindent
{\small {\sc Fakult\"at f\"ur Mathematik, Technische Universit\"at Dortmund, Vo\-gel\-poths\-weg 87, 
44221 Dortmund, Germany}

\noindent
{\em E-mail address:} tom.krantz@math.uni-dortmund.de, lschwach@math.uni-dortmund.de}

\end{document}